\documentclass[11pt,a4paper]{article}
\usepackage[utf8]{inputenc}
\usepackage[english]{babel}
\usepackage{amsmath,amsthm,enumerate}
\usepackage{amssymb}
\usepackage{graphicx}
\usepackage{tikz}
\usepackage[hmargin=2.3cm,vmargin=2.8cm]{geometry}
\usepackage{subfigure}
\usepackage[color=blue!20]{todonotes}

\newtheorem{theorem}{Theorem}

\newtheorem{proposition}[theorem]{Proposition}

\newtheorem{observation}[theorem]{Observation}
\newtheorem{corollary}[theorem]{Corollary}

\newtheorem{problem}[theorem]{Problem}
\newtheorem*{problemnonumber}{Problem}

\newtheorem{claim}{Claim}[theorem]

\newcommand{\smallqed}{{\tiny $\left(\Box\right)$}}
\newcommand{\claimproof}{\noindent\emph{Proof of claim.} }

\def\N3{N_{3^+}}
\def\vphi{\varphi}

\begin{document}

\title{Smallest $C_{2\ell +1}$-critical graphs of odd-girth $2k+1$\thanks{This work is supported by the IFCAM project Applications of graph homomorphisms (MA/IFCAM/18/39) and by the ANR project HOSIGRA (ANR-17-CE40-0022). The second author was partially financed by the French government IDEX-ISITE initiative 16-IDEX-0001 (CAP 20-25). A shorter preliminary version appeared in the proccedings of CALDAM'20~\cite{caldam}.}}

\author{Laurent Beaudou\footnote{\noindent National Research University, Higher School of Economics, 3 Kochnovsky Proezd, Moscow, Russia. E-mail: lbeaudou@hse.ru}
\and Florent Foucaud\footnote{LIMOS, CNRS UMR 6158, Universit\'e Clermont Auvergne, Aubi\`ere, France. E-mail: florent.foucaud@uca.fr}~\footnote{Univ. Bordeaux, Bordeaux INP, CNRS, LaBRI, UMR5800, F-33400 Talence, France.}~\footnote{Univ. Orl\'eans, INSA Centre Val de Loire, LIFO EA 4022, F-45067 Orl\'eans, France.}
\and Reza Naserasr\footnote{\noindent Université de Paris, CNRS, IRIF, F-75006, Paris, France. E-mail: reza@irif.fr}}

\maketitle

\begin{abstract}
Given a graph $H$, a graph $G$ is called $H$-critical if $G$ does not admit a homomorphism to $H$, but any proper subgraph of $G$ does. Observe that $K_{k-1}$-critical graphs are the standard $k$-(colour)-critical graphs. We consider questions of extremal nature previously studied for $k$-critical graphs and generalize them to $H$-critical graphs. After complete graphs, the next natural case to consider for $H$ is that of the odd-cycles. Thus, given integers $\ell$ and $k$, $\ell\geq k$, we ask: what is the smallest order of a $C_{2\ell +1}$-critical graph of odd-girth at least $2k+1$? Denoting this value by $\eta(k,C_{2\ell+1})$, we show that $\eta(k,C_{2\ell+1})=4k$ for $1\leq\ell\leq k\leq\frac{3\ell+i-3}{2}$ ($2k=i\bmod 3$) and that $\eta(3,C_5)=15$. The latter means that a smallest graph of odd-girth~$7$ not admitting a homomorphism to the $5$-cycle is of order~$15$. Computational work shows that there are exactly eleven such graphs on $15$~vertices of which only two are $C_5$-critical.

\end{abstract}

\section{Introduction}

A $k$-critical graph is a graph which is $k$-chromatic, but any of its proper subgraphs is $(k-1)$-colourable.
Extremal questions on critical graphs are a rich source of research in graph theory. Many well-known results and conjectures deal with this subject, see for example~\cite{dirac,EG19,gallai1,gallai2,NRS10}. Typical questions are for example:

\begin{problemnonumber}
What is the smallest possible order of a $k$-critical graph having a certain property, such as low clique number, high girth or high odd-girth?
\end{problemnonumber}

For example, the existence of graphs of high girth and high chromatic number, proved by Erd\H{o}s~\cite{erdos}, is a starting point of this problem. This fact implies that each of the above questions has a finite answer. The specific question of the smallest $4$-critical graph without a triangle has received considerable attention: Gr\"otzsch built a graph on $11$ vertices which is triangle-free and not $3$-colourable. Harary~\cite{H69} showed that any such graph must have at least $11$~vertices, and Chv\'atal~\cite{C74} showed that the Gr\"otzsch graph is the only one on 11 vertices.

In the context of colouring, it is of interest to consider the \emph{odd-girth}, that is, the size of a smallest odd-cycle of a graph, rather than the girth, since every graph with no odd-cycle is $2$-colourable. Extending the example of the Gr\"otzsch graph, Mycielski~\cite{M55} introduced a construction, now known as the Mycielski contruction, to increase the chromatic number without increasing the clique number.  A generalization of this construction is used to build $4$-critical graphs of high odd-girth, more precisely the \emph{generalized Mycielski construction on $C_{2k+1}$}, denoted $M_{k}(C_{2k+1})$, is a graph of odd-girth $2k+1$. Several authors (starting with Payan~\cite{P92}) showed that $M_{k}(C_{2k+1})$ is $4$-chromatic for any $k\geq 2$ and in fact $4$-critical, thus providing an upper bound of $2k^2+k+1$ for the minimum order $\eta(k)$ of a $4$-critical graph of odd-girth at least $2k+1$. (We will later generalize the function $\eta$ using the notion of homomorphism.) We refer to~\cite{GJS04,NT95,T01,Y96} for several other proofs. Among these authors, Ngoc and Tuza~\cite{NT95} asked whether the upper bound $\eta(k)\leq 2k^2+k+1$ is essentially optimal. The first quadratic lower bound for $\eta(k)$ was established in \cite{N99} by considering a maximum induced bipartite subgraph. Improving on this approach, Jiang in \cite{J01} proved $\eta(k)\geq (k-1)^2+2$.  Recently, Esperet and Stehl\'ik~\cite{ES15} proved that $2k^2+k+1$ is the correct bound for the special case of graphs where any two odd-cycles have a common vertex.\footnote{In fact, Ngoc and Tuza~\cite{NT95} equivalently asked for the value of the smallest constant $c$ such that the odd-girth of any $4$-chromatic graph of order $n$ is at most $(c+o(1))\sqrt{n}$. In this language, the tentative value $\eta(k)=2k^2+k+1$ translates to a value of $\sqrt{2n-7/4} + 1/2$ for the odd-girth of a $4$-chromatic graph of order~$n$; this is also the formulation used in~\cite{ES15}. The result of Jiang~\cite{J01} is formulated as a strict upper bound of $2\sqrt{n}+3$ for the odd-girth of a $4$-chromatic graph of order~$n$.}

The current work is a first step towards generalizing these extremal questions for $k$-critical graphs to $H$-critical graphs, defined using the terminology of homomorphisms.  
A \emph{homomorphism} of a graph $G$ to a graph $H$ is a vertex-mapping that preserves adjacency, i.e., a mapping $\psi:V(G) \to V(H)$ such that if $x$ and $y$ are adjacent in $G$, then $\psi(x)$ and $\psi(y)$ are adjacent in $H$. If there exists a homomorphism of $G$ to $H$, we may write $G \to H$ and we may say that $G$ is $H$-colourable, or that $G$ maps to $H$. In the study of homomorphisms, it is usual to work with \emph{the core} of a graph, that is, a minimal subgraph which admits a homomorphism from the graph itself. It is not difficult to show that a core of any graph is unique up to isomorphism. A graph is said to be \emph{a core} if it admits no homomorphism to a proper subgraph. We refer to the book~\cite{HNbook} for a reference on these notions.

Homomorphisms generalize proper vertex-colourings. Indeed, a homomorphism of $G$ to $K_k$ is equivalent to a $k$-colouring of $G$. However, the extension of the notion of colour-criticality to a homomorphism-based one is not much studied. As defined by Catlin~\cite{C84}, for a graph $H$, (we may assume $H$ is a core), a graph $G$ is said to be \emph{$H$-critical} if $G$ does not have a homomorphism to $H$ but any proper subgraph of $G$ does. We thus have the following.

\begin{observation}\label{obs:k-critical-is-K_k-1-critical}
A graph $G$ is $k$-critical if and only if it is $K_{k-1}$-critical.
\end{observation}  

This gives a large number of interesting extremal questions. By Observation~\ref{obs:k-critical-is-K_k-1-critical}, these questions are well-studied when $H$ is a complete graph. The next most accessible family of graphs $H$ to be considered is the one of odd-cycles. For example, the extremal question of the minimum number of edges of an $H$-critical graph on $n$ vertices has been the subject of extensive studies for $H=K_n$, see~\cite{KY14} for results and references on this subject. Recently, extensions of this study to $C_5$-critical graphs~\cite{DP17} and to $C_7$-critical graphs~\cite{PS19} were performed. We note, furthermore, that the $(2k+1)$-colouring problem is captured by the $C_{2k+1}$-colouring problem through a basic graph operation. For further details and extension of the study to signed graphs we refer to \cite{NPW21}.

In this work, we are interested in the order of a smallest $C_{2\ell +1}$-critical graph having the homomorphism property that its odd-girth is at least a given value.\footnote{This is a homomorphism property because any graph $G$ has odd-girth at least $2k+1$ if and only if $C_{2k-1}$ does not map to $G$.} We generally denote by $\eta(k,H)$ the smallest order of a graph of odd-girth at least~$k$ that has no homomorphism to $H$, and we ask the following.

\begin{problem}\label{problem}
Given positive integers $k,\ell$, what is the smallest order $\eta(k,C_{2\ell+1})$ of a $C_{2\ell +1}$-critical graph of odd-girth at least $2k+1$?
\end{problem}

In this paper, we study Problem~\ref{problem} when $\ell\geq 2$. As we discuss in Section~\ref{sec:prelim}, it follows from a theorem of Gerards~\cite{G88} that $\eta(k,C_{2\ell+1})\geq 4k$ whenever $\ell\leq k$, and $\eta(k,C_{2k+1})=4k$. We prove (in Section~\ref{sec:4k}) that, surprisingly, $\eta(k,C_{2\ell+1})=4k$ whenever $\ell\leq k\leq\frac{3\ell+i-3}{2}$ (with $2k=i\bmod 3$). We then prove (in Section~\ref{sec:eta(3,2)}) that $\eta(3,C_5)=15$. We conclude with further research questions in the last section. Table~\ref{table} summarizes the known bounds for Problem~\ref{problem} and small values of $k$ and $\ell$.

\renewcommand{\arraystretch}{0.8}

\begin{table}[ht!]
\centering
\scalebox{0.8}{
    \begin{tabular}{c||c|c|c|c|c|c|c|c}
 & $k=1$ & $k=2$ & $k=3$ & $k=4$ & $k=5$ & $k=6$ & $k=7$ & $k=8$\\
\hline
&&&&&&&&\\[-2mm]
$\ell=1$ & $4$ & $11$ & $\textbf{15}$--$22$ & $\textbf{17}$--$37$ & $\textbf{20}$--$56$ & $27$--$79$ & $38$--$106$ & $51$--$137$ \\
& & \scriptsize{\cite{H69}} & \scriptsize{[Th.~\ref{thm:OG7-order14-mappingC5}]--\cite{P92}} & \scriptsize{[Co.~\ref{cor:eta(4,1)}]--\cite{P92}} & \scriptsize{[Co.~\ref{cor:eta(k,l)>=4k}]--\cite{P92}} & \scriptsize{\cite{J01}--\cite{P92}} & \scriptsize{\cite{J01}--\cite{P92}} & \scriptsize{\cite{J01}--\cite{P92}} \\[3mm]
$\ell=2$ & $3$ & $\textbf{8}$ & $\textbf{15}$ & $\textbf{17}$--$37$ & $\textbf{20}$--$56$ & $\textbf{24}$--$79$ & $\textbf{28}$--$106$ & $\textbf{32}$--$137$ \\
& & \scriptsize{[Co.~\ref{cor:subd-K4}]} & \scriptsize{[Th.~\ref{thm:OG7-order14-mappingC5}]} & \scriptsize{[Co.~\ref{cor:eta(4,1)}]--\cite{P92}} & \scriptsize{[Co.~\ref{cor:eta(k,l)>=4k}]--\cite{P92}} & \scriptsize{[Co.~\ref{cor:eta(k,l)>=4k}]--\cite{P92}} & \scriptsize{[Co.~\ref{cor:eta(k,l)>=4k}]--\cite{P92}} & \scriptsize{[Co.~\ref{cor:eta(k,l)>=4k}]--\cite{P92}} \\[3mm]
$\ell=3$ & $3$ & $5$ & $\textbf{12}$ & $\textbf{16}$ & $\textbf{20}$--$56$ & $\textbf{24}$--$79$ & $\textbf{28}$--$106$ & $\textbf{32}$--$137$ \\
& & & \scriptsize{[Co.~\ref{cor:subd-K4}]} & \scriptsize{[Th.~\ref{thm:columns}]} & \scriptsize{[Co.~\ref{cor:eta(k,l)>=4k}]--\cite{P92}}  & \scriptsize{[Co.~\ref{cor:eta(k,l)>=4k}]--\cite{P92}} & \scriptsize{[Co.~\ref{cor:eta(k,l)>=4k}]--\cite{P92}} & \scriptsize{[Co.~\ref{cor:eta(k,l)>=4k}]--\cite{P92}} \\[3mm]
$\ell=4$ & $3$ & $5$ & $7$ & $\textbf{16}$ & $\textbf{20}$ & $\textbf{24}$--$79$ & $\textbf{28}$--$106$ & $\textbf{32}$--$137$ \\
& & & & \scriptsize{[Co.~\ref{cor:subd-K4}]} & \scriptsize{[Th.~\ref{thm:columns}]} & \scriptsize{[Co.~\ref{cor:eta(k,l)>=4k}]--\cite{P92}} & \scriptsize{[Co.~\ref{cor:eta(k,l)>=4k}]--\cite{P92}} & \scriptsize{[Co.~\ref{cor:eta(k,l)>=4k}]--\cite{P92}} \\[3mm]
$\ell=5$ & $3$ & $5$ & $7$ & $9$ & $\textbf{20}$ & $\textbf{24}$ & $\textbf{28}$ & $\textbf{32}$--$137$ \\
& & & & & \scriptsize{[Co.~\ref{cor:subd-K4}]} & \scriptsize{[Th.~\ref{thm:columns}]} & \scriptsize{[Th.~\ref{thm:columns}]} & \scriptsize{[Co.~\ref{cor:eta(k,l)>=4k}]--\cite{P92}}\\[3mm]
$\ell=6$ & $3$ & $5$ & $7$ & $9$ & $11$ & $\textbf{24}$ & $\textbf{28}$ & $\textbf{32}$ \\
& & & & & & \scriptsize{[Co.~\ref{cor:subd-K4}]} & \scriptsize{[Th.~\ref{thm:columns}]} & \scriptsize{[Th.~\ref{thm:columns}]} \\[3mm]
$\ell=7$ & $3$ & $5$ & $7$ & $9$ & $11$ & $13$ & $\textbf{28}$ & $\textbf{32}$\\
& & & & & & & \scriptsize{[Co.~\ref{cor:subd-K4}]} & \scriptsize{[Th.~\ref{thm:columns}]} \\[3mm]
$\ell=8$ & $3$ & $5$ & $7$ & $9$ & $11$ & $13$ & $15$ & $\textbf{32}$\\
& & & & & & & & \scriptsize{[Co.~\ref{cor:subd-K4}]}
    \end{tabular}
}
\caption{Known values and bounds on the smallest order $\eta(k,C_{2\ell+1})$ of a $C_{2\ell +1}$-critical graph of odd-girth at least $2k+1$. Bold values are proved in this paper.}
\label{table}
\end{table}

Note that the value of $\eta(3,C_5)$ indeed was the initial motivation of this work. In~\cite{BoundingSPG}, we use the fact that $\eta(3,C_5)=15$ to prove that if a graph $B$ of odd-girth~$7$ has the property that any series-parallel graph of odd-girth~$7$ admits a homomorphism to $B$, then $B$ has at least~$15$ vertices. The 15 vertices solution then strengthen at the same time results on circular and fractional coloring of this family of graphs. One may expect similar properties for other extremal solutions.

\section{Preliminaries}\label{sec:prelim}

This section is devoted to the introduction of useful preliminary notions and results.

\paragraph{Circular chromatic number.}

We recall some basic notions related to circular colourings. For a survey on the matter, consult~\cite{Z01survey}.
Given two integers $p$ and $q$ with $gcd(p,q)=1$, the \emph{circular clique} $C(p,q)$ is the graph 
on vertex set $\{0,\ldots, p-1\}$ with $i$ adjacent to $j$ if and only if $q\leq |i-j|\leq p-q$.

A homomorphism of a graph $G$ to $C(p,q)$ is called a $(p,q)$-colouring, and the \emph{circular chromatic number}  of $G$, denoted $\chi_c(G)$, is the smallest rational $\dfrac{p}{q}$ such that $G$ has a $(p,q)$-colouring. Since $C(p,1)$ is the complete graph $K_{p}$, we have $\chi_c(G)\leq \chi(G)$.
On the other hand $C(2\ell+1,\ell)$ is the cycle $C_{2\ell +1}$. Thus $C_{2\ell +1}$-colourability is about deciding whether $\chi_c(G)\leq 2+\dfrac{1}{\ell}$.

It is a well-known fact that $C(p,q)\to C(r,s)$ if and only if $\dfrac{p}{q} \leq \dfrac{r}{s}$ (e.g. see \cite{Z01survey}), in particular we will use the fact that $C(12,5)$ has a homomorphism to $C(5,2)$, that is, $C_5$.

\paragraph{Odd-$K_4$'s and a theorem of Gerards.}
The following notion will be central in our proofs. An \emph{odd-$K_4$} is a subdivision of the complete graph $K_4$ where each of the four triangles of $K_4$ has become an odd-cycle~\cite{G88}. Furthermore, we call it a \emph{$(2k+1)$-odd-$K_4$} if each such cycle has length exactly $2k+1$.
Since subdivided triangles are the only odd-cycles of an odd-$K_4$, the odd-girth of a $(2k+1)$-odd-$K_4$ is~$2k+1$. The following is an easy fact about odd-$K_4$'s whose proof we leave as an exercise.

\begin{proposition}\label{prop:oddK4}
Let $K$ be an odd-$K_4$ of odd-girth at least~$2k+1$. Then, $K$ has order at least $4k$, with equality if and only if $K$ is a $(2k+1)$-odd-$K_4$. 
Furthermore, in the last case any two disjoint edges of $K_4$ are subdivided the same number of times when constructing $K$.
\end{proposition}

A $(2k+1)$-odd-$K_4$ is, more precisely, referred to as an \emph{$(a,b,c)$-odd-$K_4$} if three edges of a triangle of $K_4$ are subdivided into paths of length $a$, $b$ and $c$ respectively (by Proposition~\ref{prop:oddK4} this is true for all four triangles). Note that while the terms ``odd-$K_4$'' or ``$(2k+1)$-odd-$K_4$'' refer to many non-isomorphic graphs, an $(a,b,c)$-odd-$K_4$ ($a+b+c=2k+1$) is unique up to a relabeling of vertices. See~\cite[Section 4.2.1]{NW21} for some homomomorphism properties of odd-$K_4$'s.

An \emph{odd-$K_3^2$} is a graph obtained from three disjoint odd-cycles and three disjoint paths (possibly of length~$0$) joining each pair of cycles~\cite{G88}. Thus, in such a graph, any two of the three cycles have at most one vertex in common (if the path joining them has length~$0$). Hence, an odd-$K_3^2$ of odd-girth at least~$2k+1$ has order at least~$6k$.

In a graph $G$, a \emph{thread} is a path in $G$ where the internal vertices have degree~$2$ in $G$. For example, in an odd-$K_4$ or in an odd-$K_3^2$, the paths of degree~$2$-vertices joining two vertices of degree at least~$3$ are threads.

\begin{theorem}[Gerards \cite{G88}]\label{thm:gerards}
If $G$ has neither an odd-$K_4$ nor an odd-$K_3^2$ as a subgraph, then it admits a homomorphism to its shortest odd-cycle. 
\end{theorem}

We have the following corollary of Proposition~\ref{prop:oddK4} and Theorem~\ref{thm:gerards}. 

\begin{corollary}\label{cor:subd-K4}
For any positive integer $k$, we have $\eta(k,C_{2k+1})=4k$.
\end{corollary}
\begin{proof}
  Consider a $C_{2k+1}$-critical graph $G$ of odd-girth~$2k+1$. It follows from Theorem~\ref{thm:gerards} that $G$ contains either an odd-$K_4$, or an odd-$K_3^2$. If it contains the latter, then $G$ has at least $6k$ vertices. Otherwise, $G$ must contain an odd-$K_4$ of odd-girth at least~$2k+1$, and then by Proposition~\ref{prop:oddK4}, $G$ has at least $4k$ vertices. This shows that $\eta(k,C_{2k+1})\geq 4k$.

  Moreover, no $(2k+1)$-odd-$K_4$ admits a homomorphism to $C_{2k+1}$. Indeed, assume by contradiction that there is such a homomorphism $\varphi$ of a $(2k+1)$-odd-$K_4$, $G$, to $C_{2k+1}$, and let $x,y,z,t$ be the vertices of degree~$3$ of $G$. Consider the $(2k+1)$-cycle $C$ that passes through $x,y,z$. This cycle must be mapped surjectively to $C_{2k+1}$ by $\varphi$. Thus, $t$ must be identified with some vertex of $C$ by $\varphi$. If it is identified with a vertex of the $xy$-thread of $G$, then the $(2k+1)$-cycle going through $x,y,t$ would not be mapped surjectively to $C$ by $\varphi$, a contradiction. The same argument applies to the $xz$-thread and the $yz$-thread, showing that $t$ cannot be mapped and thus $\varphi$ does not exist.

  Since any $(2k+1)$-odd-$K_4$ has order~$4k$ by Proposition~\ref{prop:oddK4}, the latter paragraph shows that $\eta(k,C_{2k+1})\leq 4k$.
\end{proof}

Since $C_{2\ell +3}$ maps to $C_{2\ell +1}$ and by transitivity of homomorphisms, a graph with no homomorphism to $C_{2\ell +1}$ also has no homomorphism to $C_{2\ell +3}$. Thus:

\begin{observation}\label{obs:monotone-columns}
Let $k,\ell$ be two positive integers. We have $\eta(k,C_{2\ell+1})\geq\eta(k,C_{2\ell+3})$.
\end{observation}

We obtain this immediate consequence of Corollary~\ref{cor:subd-K4} and Observation~\ref{obs:monotone-columns}:

\begin{corollary}\label{cor:eta(k,l)>=4k}
For any two integers $k,\ell$ with $k\geq \ell\geq 1$, we have $\eta(k,C_{2\ell+1})\geq 4k$.
\end{corollary}

\section{Rows of Table~\ref{table}: $\eta(k,C_{2\ell+1})$ for fixed $\ell$}\label{sec:RowsColumns}\label{sec:4k}

In this section, we study the behavior of $\eta(k,C_{2\ell+1})$ when $\ell$ is a fixed value, that is, the behavior of each row of Table~\ref{table}. Note again that whenever $\ell\geq k+1$, we have $\eta(k,C_{2\ell+1})=2k+1$. 
As mentioned before, the first row (i.e. $\ell=1$) is about the smallest order of a $4$-critical graph of odd-girth $2k+1$ and we know $\eta(k,C_3)=\Theta(k^2)$~\cite{N99}.

It is not difficult to observe that for fixed $\ell$, the function $\eta(k,C_{2\ell+1})$ is strictly increasing, in fact with a little bit of effort we can even show the following.

\begin{proposition}\label{prop:etaIncreasing}
For $k\geq \ell$, we have $\eta(k+1,C_{2\ell+1})\geq \eta(k,C_{2\ell+1})+2$.
\end{proposition}
\begin{proof}
Let $G$ be a $C_{2\ell +1}$-critical graph of odd-girth~$2k+3$ and order $\eta(k+1,C_{2\ell+1})$. Consider any $(2k+3)$-cycle $v_0\cdots v_{2k+2}$ of $G$, and build a smaller graph by identifying $v_0$ with $v_2$ and $v_1$ with $v_3$. It is not difficult to check that the resulting graph has odd-girth exactly $2k+1$ and does not map to $C_{2\ell +1}$ (otherwise, $G$ would), proving the claim.
\end{proof}

While it is likely that for a fixed $\ell$, the value of $\eta(k,C_{2\ell+1})$ grows quadratically in terms of $k$, we show, somewhat surprisingly, that at least just after the threshold of $k=\ell$, the function $\eta(k,C_{2\ell+1})$ only increases by $4$ when $\ell$ increases by~$1$, implying that Proposition~\ref{prop:etaIncreasing} cannot be improved much in this formulation.
More precisely, we have the following theorem.

\begin{theorem}\label{thm:columns}
For any $k,\ell\geq 3$ and $\ell\leq k\leq\frac{3\ell+i-3}{2}$ (where $2k=i\bmod 3$), we have $\eta(k,C_{2\ell+1})=4k$.
\end{theorem}

To prove this theorem, we give a family of $C_{2\ell +1}$-critical odd-$K_4$'s which are of odd-girth $2k+1$. This is done in the next theorem, after which we give a proof of Theorem~\ref{thm:columns}.

\begin{theorem}\label{thm:4k}
Let $p\geq 3$ be an integer. If $p$ is odd, any $(a,b,c)$-odd-$K_4$ with $(a,b,c)\in\{(p-1,p-1,p),(p,p,p)\}$ has no homomorphism to $C_{2p+1}$. If $p$ is even, any $(p-1,p,p)$-odd-$K_4$ has no homomorphism to $C_{2p+1}$.
\end{theorem}
\begin{proof}
Let $(a,b,c)\in \{(p-1,p-1,p),(p,p,p),(p-1,p,p)\}$ and let $K$ be an $(a,b,c)$-odd-$K_4$. Let $t$, $u$, $v$, $w$ be the vertices of degree~$3$ in $K$ with the $tu$-thread of length~$a$, the $uv$-thread of length~$b$ and the $tv$-thread of length~$c$. We now distinguish two cases depending on the parity of $p$ and the values of $(a,b,c)$.

\medskip\noindent\emph{Case 1.} Assume that $p$ is odd and $(a,b,c)\in\{(p-1,p-1,p),(p,p,p)\}$. By contradiction, we assume that there is a homomorphism $h$ of $K$ to $C_{2p+1}$. Then, the cycle $C_{tvw}$ formed by the union of the $tv$-thread, the $vw$-thread and the $tw$-thread is an odd-cycle of length $a+b+c$. Therefore, its mapping by $h$ to $C_{2p+1}$ must be \emph{onto}. Thus, $u$ has the same image by $h$ as some vertex $u'$ of $C_{tvw}$. 
Note that $u'$ is not one of $t$, $v$ or $w$, indeed by identifying $u$ with any of these vertices we obtain a graph containing an odd-cycle of length~$p$ or~$2p-1$; thus, this identification cannot be extended to a homomorphism to $C_{2p+1}$. Therefore, $u'$ is an internal vertex of one of the three maximal threads in $C_{tvw}$. Let $C_u$ be the odd-cycle of length $a+b+c$ containing $u$ and $u'$. After identifying $u$ and $u'$, $C_u$ is transformed into two cycles, one of them being odd. If $(a,b,c)=(p,p,p)$, then $C_{u}$ has length~$3p$. Then, the two newly created cycles have length at least~$p+1$, and thus at most $2p-1$. If $(a,b,c)=(p-1,p-1,p)$, then $C_{u}$ has length $3p-2$, and the two cycles have length at least~$p$ and at most $2p-2$. In both cases, we have created an odd-cycle of length at most $2p-1$. Hence, this identification cannot be extended to a homomorphism to $C_{2p+1}$, a contradiction.

\medskip\noindent\emph{Case 2.} Assume that $p$ is even and $(a,b,c)=(p-1,p,p)$, and that $h$ is a homomorphism of $K$ to $C_{2p+1}$. Again, the image of $C_{tvw}$ by $h$ is onto, and $u$ has the same image as some vertex $u'$ of $C_{tvw}$. If $u'=t$, identifying $u$ and $u'$ produces an odd $(p-1)$-cycle, a contradiction. If $u'\in\{v,w\}$, then we get a $(2p-1$)-cycle, a contradiction. Thus, $u'$ is an internal vertex of one of the three maximal threads in $C_{tvw}$. Let $C_u$ be the odd-cycle of length~$3p-1$ containing $u$ and $u'$. As in Case~$1$, after identifying $u$ and $u'$, $C_{u}$ is transformed into two cycles, each of length at least~$p$ and at most~$2p-1$; one of them is odd, a contradiction.
\end{proof}

We note that Theorem~\ref{thm:4k} is tight, in the sense that if $p$ is odd and $(a,b,c)\in\{(p-1,p-1,p),(p,p,p)\}$ or if $p$ is even and $(a,b,c)=(p-1,p,p)$, then an $(a,b,c)$-odd-$K_4$ has a homomorphism to $C_{2p-1}$.

We can now prove Theorem~\ref{thm:columns}.

\begin{proof}[Proof of Theorem~\ref{thm:columns}]
By Corollary~\ref{cor:eta(k,l)>=4k}, we know that $\eta(k,C_{2\ell+1})\geq 4k$.
We now prove the upper bound. Recall that $\eta(k,C_{2\ell+1})\leq\eta(k,C_{2\ell-1})$. If $2k \equiv 0\pmod 3$, then $p=\frac{2k+3}{3}$ is an odd integer, and $p\leq \ell$. By Theorem~\ref{thm:4k}, a $(p-1,p-1,p)$-odd-$K_4$, which has order $6p-6=4k$, has no homomorphism to $C_{2p+1}$, and thus $\eta(k,C_{2\ell+1})\leq\eta(k,C_{2p+1})\leq 4k$. Similarly, if $2k\equiv 1\pmod 3$, then $p=\frac{2k+2}{3}$ is an even integer, and $p\leq \ell$. By Theorem~\ref{thm:4k}, a $(p-1,p,p)$-odd-$K_4$, which has order $6p-4=4k$, has no homomorphism to $C_{2p+1}$, and thus $\eta(k,C_{2\ell+1})\leq\eta(k,C_{2p+1})\leq 4k$. Finally, if $2k\equiv 2\pmod 3$, then $p=\frac{2k+1}{3}$ is an odd integer, and $p\leq \ell$. By Theorem~\ref{thm:4k}, a $(p,p,p)$-odd-$K_4$, which has order $6p-2=4k$, has no homomorphism to $C_{2p+1}$, and thus $\eta(k,C_{2\ell+1})\leq\eta(k,C_{2p+1})\leq 4k$.
\end{proof}

\section{The value of $\eta(3,C_5)$}\label{sec:eta(3,2)}

We now determine $\eta(3,C_5)$, which is not covered by Theorem~\ref{thm:4k}. By Corollary~\ref{cor:eta(k,l)>=4k}, we know that $\eta(3,C_5)\geq 12$. In fact, we will show that $\eta(3,C_5)=15$.

Examples of graphs of odd-girth 7 on 15 vertices are given in Figure~\ref{fig:2graphsOrder15}. In fact, 
using a computer search, Gordon Royle (private communication, 2016) has verified that these two graphs are the only two $C_5$-critical graphs of order $15$ and odd-girth~$7$. 

\begin{figure}[!htpb]
\begin{center}
\subfigure{\scalebox{1}{\begin{tikzpicture}[join=bevel,inner sep=.5mm,scale=0.7]

   \foreach \i in {1,...,8} 
		{
			\draw[rotate=360/8*\i] (0, 4) node[circle, draw=black!80, inner sep=0mm, minimum size=1.8mm] (u\i){};
		}

\foreach \j in {1,...,7} 
		{
			\draw[rotate=360/7*(\j+1)] (0, 2.4) node[circle, draw=black!80, inner sep=0mm, minimum size=1.8mm] (v\j){};
		}

	\foreach \i/\j in {1/2,2/3,3/4,4/5,5/6,6/7,7/8,8/1}
		{
			\draw  [line width=0.4mm] (u\i) -- (u\j);
		}

	\foreach \i/\j in {1/2,2/3,3/4,4/7,7/6,6/5,5/1}
		{
			\draw  [line width=0.4mm] (v\i) -- (v\j);
		}
		
\foreach \i/\j in {1/7,7/5,3/1,5/4}
		{
			\draw  [line width=0.4mm] (u\i) -- (v\j);
		}
\end{tikzpicture}}}\qquad
\subfigure{\scalebox{1}{\begin{tikzpicture}[join=bevel,inner sep=0.5mm,scale=0.7]

   \foreach \i in {1,...,8} 
		{
			\draw[rotate=360/8*\i] (0, 4) node[circle, draw=black!80, inner sep=0mm, minimum size=1.8mm] (u\i){};
		}

\foreach \j in {1,...,7} 
		{
			\draw[rotate=360/7*(\j+1)-6] (0, 2.4) node[circle, draw=black!80, inner sep=0mm, minimum size=1.8mm] (v\j){};
		}

	\foreach \i/\j in {1/2,2/3,3/4,4/5,5/6,6/7,7/8,8/1}
		{
			\draw  [line width=0.4mm] (u\i) -- (u\j);
		}

	\foreach \i/\j in {1/2,2/3,3/4,4/5,5/6,6/7,7/1}
		{
			\draw  [line width=0.4mm] (v\i) -- (v\j);
		}

\foreach \i/\j in {8/5,7/4,2/2, 3/3,5/7}
		{
			\draw  [line width=0.4mm] (u\i) -- (v\j);
		}
	
\end{tikzpicture}}}
\end{center}
\caption{The two $C_5$-critical graphs of order~$15$ and odd-girth~$7$.}
\label{fig:2graphsOrder15}
\end{figure}

One may add a few edges to each of these two graphs without creating a shorter odd cycle. G. Royle reports that there are eleven non isomorphic graphs built in this way. For the most symmetric of all these graphs two different presentations are given in Figure~\ref{fig:SymmetricGraphsOrder15}. 

\begin{figure}[!htpb]
\begin{center}
\subfigure{\scalebox{1}{\begin{tikzpicture}[join=bevel,inner sep=.5mm,scale=0.7]

   \foreach \i in {1,...,9} 
		{
			\draw[rotate=360/9*\i] (0, 4) node[circle, draw=black!80, inner sep=0mm, minimum size=1.8mm] (u\i){};
		}

\foreach \j in {1,...,6} 
		{
			\draw[rotate=360/6*(\j+1)-9] (0, 2.4) node[circle, draw=black!80, inner sep=0mm, minimum size=1.8mm] (v\j){};
		}

	\foreach \i/\j in {1/2,2/3,3/4,4/5,5/6,6/7,7/8,8/9,9/1}
		{
			\draw  [line width=0.4mm] (u\i) -- (u\j);
		}

	\foreach \i/\j in {1/2,3/4,5/6, 1/4,2/5,3/6}
		{
			\draw  [line width=0.4mm] (v\i) -- (v\j);
		}
		
\foreach \i/\j in {3/1,4/2,6/3,7/4,9/5,1/6}
		{
			\draw  [line width=0.4mm] (u\i) -- (v\j);
		}
\end{tikzpicture}}}\qquad
\subfigure{\scalebox{1}{\begin{tikzpicture}[join=bevel,inner sep=0.5mm,scale=0.7]

   \foreach \i in {1,...,6} 
		{
			\draw[rotate=360/6*\i] (0, 2.4) node[circle, draw=black!80, inner sep=0mm, minimum size=1.8mm] (u\i){};
		}

\foreach \i in {1,...,6} 
		{
			\draw[rotate=360/6*\i] (0, 4) node[circle, draw=black!80, inner sep=0mm, minimum size=1.8mm] (v\i){};
		}

\foreach \i in {1,2,3} 
		{
			\draw[rotate=360/3*\i+30] (.5, .5) node[circle, draw=black!80, inner sep=0mm, minimum size=1.8mm] (w\i){};
		}

	\foreach \i/\j in {1/2,3/4,5/6}
		{
			\draw  [line width=0.4mm] (u\i) -- (u\j);
		}

	\foreach \i/\j in {1/2,2/3,3/4,4/5,5/6,6/1}
		{
			\draw  [line width=0.4mm] (v\i) -- (v\j);
		}

\foreach \i in {1,...,6}
		{
			\draw  [line width=0.4mm] (u\i) -- (v\i);
		}
		
\foreach \i/\j in {1/3,1/6, 2/2, 2/5,3/1,3/4}
		{
			\draw  [line width=0.4mm] (w\i) -- (u\j);
		}
	
\end{tikzpicture}}}
\end{center}
\caption{Two drawings of a symmetric graph of odd-girth~$7$ which does not map to $C_5$.}
\label{fig:SymmetricGraphsOrder15}
\end{figure}

This implies that $\eta(3,C_5)\leq 15$. Next, we prove that this upper bound is tight.

\begin{theorem}\label{thm:OG7-order14-mappingC5}\label{thm:15}
Any graph $G$ of order at most $14$ and odd-girth at least~$7$ admits a homomorphism to $C_5$, and thus $\eta(3,C_5)=15$.
\end{theorem}
\begin{proof}
We consider a $C_5$ on the vertex set $\{0,1,2,3,4\}$  where vertex $i$ is adjacent to vertices $i+1$ and $i-1$ (modulo $5$). Thus, in the following, to give a $C_5$-colouring
we will give a colouring using elements of $\{0,1,2,3,4\}$ where adjacent pairs are mapped into (cyclically) consecutive elements of this set.

Given a graph $G$ and a vertex $v$ of it, we partition $V(G)$ into four sets
$\{v\}, N_1(v), N_2(v)$ and $\N3(v)$ where $N_1(v)$ (respectively $N_2(v)$) 
designates the set of vertices at distance exactly~1 (respectively~2) of $v$, 
and $\N3(v)$ the vertices at distance~3 or more of $v$. Note that because of the odd-girth requirement for $G$, $N_1(v)$ and $N_2(v)$ are independent sets.

A proper 3-colouring of $G[\N3(v)]$ using colours $c_1, c_2$ and $c_3$ is said to be \emph{$v$-special} if:
	\begin{itemize}
		\item[(i)] each vertex with colour $c_3$ is an isolated vertex of $G[\N3(v)]$, 
		\item[(ii)] no vertex from $N_2(v)$ sees both colours $c_1$ and $c_2$.
	\end{itemize}
	
	A key observation is the following: given any graph $G$, if for some
	vertex $v$ of $G$, there exists a $v$-special colouring of $G[\N3(v)]$,
	then $G$ maps to $C_5$. Such a homomorphism is given by
	mapping $c_1$-vertices to~0, $c_2$-vertices to~1 and $c_3$-vertices
	to~3, and then extending as follows:
	\begin{itemize}
		\item for any vertex $u$ in $N_2(v)$, if $u$ has a $c_1$-neighbour,
		map it to~4; otherwise, map it to~2,
		\item all vertices of $N_1(v)$ are mapped to~3,
		\item vertex $v$ is mapped to~2 or~4.
	\end{itemize}
	
	Now, let $G$ be a minimal counterexample to Theorem~\ref{thm:15}: $G$ has odd-girth at least~$7$, has order at most~$14$, but does not map to $C_5$. Note that $G$ is connected, indeed $G$ does not map to $C_5$ if and only if one of its components does not map to $C_5$.

        We
	first collect a few properties of $G$. The previous paragraph allows
	us to state our first claim.
	
	\begin{claim}
		\label{clm:nospecial}
		For no vertex $v$ of $G$ there is a $v$-special colouring of $G[\N3(v)]$.
	\end{claim}
	
	Since $G$ is a minimal counterexample, it cannot map to a subgraph of itself (which would be a smaller counterexample):
        
	\begin{claim}
		\label{clm:core}\label{clm:diffneighbours}
		Graph $G$ is a core. In particular, for any two vertices $u$ and $v$ of $G$, $N(u) \not\subseteq N(v)$.
	\end{claim}

        Recall that a \emph{walk} between two vertices $u$ and $v$ is a sequence of (not necessarily distinct) vertices starting with $u$ and ending with $v$, where two consecutve vertices in the sequence are adjacent. A walk between $u$ and $v$ is a \emph{$uv$-walk}, and a \emph{$k$-walk} is a walk with $k+1$ vertices counting multiplicity.
	
	\begin{claim}
		\label{clm:5walk}
		For any two distinct vertices $u$ and $v$ of $G$, there is a $uv$-walk of length~5.
	\end{claim}
	\claimproof If not, identifying $u$ and $v$ would result in a
	smaller graph of yet odd-girth~7 which does not map to $C_5$, contradicting the minimality of $G$.  \smallqed
	
	\begin{claim}
		\label{clm:4threads}\label{clm:nothread}
		Graph $G$ has no thread of length~4 or more.
	\end{claim}
	\claimproof Once again, by minimality of $G$, if we remove a
	thread of length~4, the resulting graph maps to $C_5$. But
	since there is a walk of length~4 between any two vertices of $C_5$,
	this mapping could easily be extended to $G$. \smallqed\medskip

        Now we can state a more difficult claim.
	
	\begin{claim}
		\label{clm:(<=3,<=4)}
		There is no vertex of $G$ of degree~4 or more, nor a vertex of degree exactly~3 with
		a second neighbourhood of size~5 or more.
	\end{claim}
	\claimproof For a contradiction, suppose that a vertex $v$ has degree~4 or
	more, or has degree~3 and a second neighbourhood of size~5 or more. 
	
	By Claim~\ref{clm:diffneighbours}, the neighbours of $v$ should have
	pairwise distinct neighbourhoods, so that even if $v$ has degree~$4$ or
	more, we must have $|N_2(v)| \geq 4$. Thus, by a counting argument
	(recall that $G$ has at most 14 vertices),
	$\N3(v)$ has size at most~5. Since $G$ has odd-girth~7, this means
	$G[\N3(v)]$ is bipartite. 
	
	Suppose $G[\N3(v)]$ has at most one non-trivial connected
	component. Consider any proper $3$-colouring of $G[\N3(v)]$ such that
	colours $c_1$ and $c_2$ are used for the non-trivial connected
	component, and colour $c_3$ is used for isolated vertices. Then, no
	vertex of $N_2(v)$ can see both colours $c_1$ and $c_2$, as this would
	result in a short odd-cycle in $G$. Thus, any such colouring is
	$v$-special. Hence, by Claim~\ref{clm:nospecial}, we derive that $G[\N3(v)]$ has at least two
	non-trivial components. Since it has order at most~5, it must have
	exactly two.
	
	Assume now that both non-trivial connected components are isomorphic
	to $K_2$. Consider all proper $3$-colourings of $G[\N3(v)]$ such that
	colours $c_1$ and $c_2$ are used for the copies of $K_2$ (the
	potentially remaining vertex being coloured with $c_3$). Considering Claim~\ref{clm:nospecial}, one may easily check that if none of these colorings is $v$-special, then there is a short odd-cycle in
	$G$, which is a contradiction.
	
	Hence, $G[\N3(v)]$ is isomorphic to the disjoint union of
        $K_2$ and $K_{2,1}$. Let $u_1$ and $u_2$ be the vertices of
        $K_2$, and $u_3, u_4$ and $u_5$ be the vertices of $K_{2,1}$
        such that $u_4$ is the central vertex. 
	
	Let $\vphi_1$ and $\vphi_2$ be two proper 2-colourings of $G[\N3(v)]$
	as follows: (i) $\vphi_1(u_1) = \vphi_1(u_3) = \vphi_1(u_5) = c_1$ and
	 $\vphi_1(u_2) = \vphi_1(u_4) = c_2$, and (ii) $\vphi_2(u_2) = \vphi_2(u_3) =
	\vphi_2(u_5) = c_1$ and $\vphi_2(u_1) = \vphi_2(u_4) = c_2$ (see Figure~\ref{fig:2-colorings}).

\begin{figure}[!htpb]
\begin{center}
\subfigure[Colouring $\vphi_1$]{\scalebox{1}{\begin{tikzpicture}[join=bevel,inner sep=0.5mm,scale=0.7]
      \node[draw,shape=circle](u1) at (0,0) {\tiny$u_1$};\draw (u1) node[above=5mm] {$c_1$};  
      \node[draw,shape=circle](u2) at (2,0) {\tiny$u_2$}; \draw (u2) node[above=5mm] {$c_2$}; 
      \node[draw,shape=circle](u3) at (4,0) {\tiny$u_3$};\draw (u3) node[above=5mm] {$c_1$};
      \node[draw,shape=circle](u4) at (6,0) {\tiny$u_4$}; \draw (u4) node[above=5mm] {$c_2$};
      \node[draw,shape=circle](u5) at (8,0) {\tiny$u_5$}; \draw (u5) node[above=5mm] {$c_1$};     

\draw[line width=1pt] (u1)--(u2) (u3)--(u4)--(u5);

\end{tikzpicture}}}\qquad\qquad\qquad
\subfigure[Colouring $\vphi_2$]{\scalebox{1}{\begin{tikzpicture}[join=bevel,inner sep=0.5mm,scale=0.7]

      \node[draw,shape=circle](u1) at (0,0) {\tiny $u_1$};\draw (u1) node[above=5mm] {$c_2$};  
      \node[draw,shape=circle](u2) at (2,0) {\tiny$u_2$}; \draw (u2) node[above=5mm] {$c_1$}; 
      \node[draw,shape=circle](u3) at (4,0) {\tiny$u_3$};\draw (u3) node[above=5mm] {$c_1$};
      \node[draw,shape=circle](u4) at (6,0) {\tiny$u_4$}; \draw (u4) node[above=5mm] {$c_2$};
      \node[draw,shape=circle](u5) at (8,0) {\tiny$u_5$}; \draw (u5) node[above=5mm] {$c_1$};     

\draw[line width=1pt] (u1)--(u2) (u3)--(u4)--(u5);
\end{tikzpicture}}}
\end{center}
\caption{The two $2$-colorings of $\N3(v)$ in the proof of Claim~\ref{clm:(<=3,<=4)}.}
\label{fig:2-colorings}
\end{figure}

Since $\vphi_1$ is not $v$-special, there is either a vertex $t_1$ adjacent	to $u_2$ and $u_3$ (by considering the symmetry of $u_3$ and $u_5$), or a vertex $t_1'$ adjacent to $u_1$ and $u_4$. Similarly, since $\vphi_2$ is not $v$-special, either there is a vertex $t_2$ adjacent to $u_1$	and one to $u_3$ and $u_5$, or there is a vertex $t_2'$ adjacent to $u_2$ and $u_4$. If $t_1'$ exists, then in the subgraph induced by $t_1'$, $t_2$ or $t_2'$, whichever exists, and their neighbors we will find an odd-cycle of length at most 5. Similarly, $t_2'$ does not exist. Thus, $t_1$ and $t_2$ must exist. Moreover, for a similar reason, $t_2$ is not a a neighbour of $u_3$ and, therefor, it is a neighbor of $u_5$.
	
	Next, we show that $u_4$ has degree at least~3. Suppose not,
        then it has degree exactly~2. Let $\vphi_3$ be a partial
        $C_5$-colouring of $G$ defined as follows: $\vphi_3(u_1) =
        \vphi_3(u_5) = 0, \vphi_3(u_2) = 1, \vphi_3(u_3) = 3$ and
        $\vphi_3(u_4) = 4$. Then, no vertex in $N_2(v)$ sees both~0
        and~1 (by odd-girth arguments). Thus, we can extend $\vphi_3$
        to $N_2(v)$ using only colours~2 and~4 on these
        vertices. Then, all vertices of $N_1(v)$ can be mapped to~3
        and $v$ can be mapped to~2 and $G\to C_5$, a contradiction.
	
	Hence, there exists a vertex $t_3$ in $N_2(v)$ which is
        adjacent to $u_4$. Note that, by the odd-girth condition,
        $t_3$ has no other neighbour in $G[\N3(v)]$ and, in
        particular, it must be distinct from $t_1$ and $t_2$. Vertices
        $t_1, t_2$ and $t_3$ are in $N_2(v)$, so there are vertices
        $s_1$, $s_2$ and $s_3$ in $N_1(v)$ such that $s_i$ is adjacent
        to $t_i$ for $i$ between~1 and~3. Moreover, vertices $t_1,
        t_2$ and $t_3$ are pairwise connected by a path of
        length~3. Therefore, their neighbourhoods cannot intersect, so
        that vertices $s_1, s_2$ and $s_3$ are distinct.
	
	Now, consider the partial $C_5$-colouring $\vphi_3$ again. We may
	extend $\vphi_3$ to $N_2(v)$ by assigning colour~0 to neighbours of
	$u_4$, colour~4 to neighbours of $u_1$ and $u_5$, and colour~2 to the
	rest. If no vertex of $N_1(v)$ sees both colours~0 and~4 in $N_2(v)$,
	we may colour $N_1(v)$ with~1 and~3 and colour $v$ with~2, which is a
	contradiction. Thus, there exists some vertex $x$ in $N_1(v)$ seeing
	both colours~0 and~4 in $N_2(v)$. The only vertices with colour~0 in
	$N_2(v)$ are neighbours of $u_4$ so that $x$ must be at distance~2
	from $u_4$. Since there is no short odd-cycle in $G$, vertex $x$
	cannot be at distance~2 from $u_5$. Thus, it is at distance~2 from
	$u_1$. Let $t_4$ be the middle vertex of this path from $x$ to
	$u_1$. Now $t_4$ is a neighbour of $u_1$ which is distinct from $t_1,
	t_2$ and $t_3$. By the symmetry between $u_1$ and $u_2$, there must be a
	fifth vertex $t_5$ in $N_2(v)$ which is a neighbour of $u_2$ and
	distinct from $t_1,t_2,t_3$ and $t_4$. Moreover, $t_5$ has a neighbour
	$y$ in $N_1(v)$ that is at distance~2 from $u_4$. We can readily check
	that $y$ is distinct from $x, s_1$ and $s_2$. Thus, $N_1(v)$ has size
	at least~4 and $N_2(v)$ has size at least~5, which is a contradiction
	with the order of $G$ (which should be at most 14). This concludes the
	proof of Claim~\ref{clm:(<=3,<=4)}. \smallqed


\begin{claim}\label{clm:no6-cycle}
$G$ contains no $6$-cycle.
\end{claim}
\claimproof Suppose, by contradiction, that $G$ contains a $6$-cycle $C:v_0,\ldots, v_{5}$. For a pair $v_i$ and $v_{i+2}$ (addition in indices is done modulo $6$) of vertices of $C$, the $5$-walk connecting them (see Claim~\ref{clm:5walk}) is necessarily a $5$-path because of the odd-girth condition, and we denote it by $P^i$. Furthermore, at most one inner-vertex of $P^i$ may belong to $C$, and if it does, it must be a neighbour of $v_i$ or $v_{i+2}$ (one can check that otherwise, there is a short odd-cycle in $G$).

Assume first that none of the six paths $P^i$ ($0\leq i\leq 5$) has any inner-vertex on $C$. In this case, by Claim~\ref{clm:(<=3,<=4)}, we observe that the neighbours of $v_i$ in $P^i$ and $P^{i+4}$ (addition in superscripts is done modulo $6$) are the same. Let $v_i'$ be this neighbour of $v_i$.

Vertices $v'_i$, $i=0,1, \ldots, 5$ are all distinct, as otherwise we have a short odd-cycle in $G$. Let $x$ and $y$ be the two internal vertices of $P^0$ 
distinct from  $v'_0$ and $v'_2$. By our assumption, $x$ and $y$ are distinct from vertices of $C$. We claim that they are also distinct from $v_i'$, $i=0, \ldots, 5$. Vertex $x$ is indeed distinct from $v_0'$ and $v_2'$ by assumption. It is distinct from $v_1',v_3'$ and $v_4'$ as otherwise there will be a short odd-cycle. By the symmetry of $x$ and $y$ and by the same arguments, $y$ is distinct from $v_0', v_1', v_2', v_4', v_5'$. Moreover, we cannot simultaneously have $x=v_5'$ and $y=v_3'$, otherwise we get a $5$-cycle. Finally, if we have $x=v_5'$ then $\{v_0', v_1, y, v_3, v_4'\}\subseteq N_2(v_5)$ and $d(v_5)\geq 3$, contradicting Claim~\ref{clm:(<=3,<=4)}.
As a result, since $|V(G)|\leq 14$, $x$ and $y$ are internal vertices of all $P^i$'s. But then it easy to find a short odd-cycle.

Hence, we may assume that some path $P^i$, say $P^1$ without loss of generality, has (exactly) one inner-vertex on $C$. Without loss of generality, say $v'_1=v_0$ (recall that it is not possible that a vertex of $P^1$ that is not a neighbour of $v_1$ or $v_3$ is on $C$). Let $v_0x_1x_2x_3v_3$ be the $4$-path connecting $v_0$ and $v_3$ (recall that at most one inner-vertex of $P^i$ belongs to $C$, so $x_i\notin C$ for $i=1,2,3$). We next assume that $P^5$ does not have any inner-vertex in $C$. Then, no vertex of $P^5$ is a vertex from $\{x_1,x_2,x_3\}$, for otherwise we have a short odd-cycle in $G$. But then, $v_0$ violates Claim~\ref{clm:(<=3,<=4)} as $\{v_5',v_4,x_2,v_2,v_1'\}\subseteq N_2(v_0)$. Therefore, an inner-vertex of $P^5$ lies on $C$. Considering the odd-girth conditions and Claim~\ref{clm:(<=3,<=4)}, either $v_5'=v_4$, or the neighbour of $v_1$ on $P^5$ is $v_2$. As both cases are symmetric with respect to our assumptions so far, we may assume that the second case holds without loss of generality. Thus, we have $P^5:v_5y_1y_2y_3v_2v_1$. Then, $P^1$ and $P^5$ are vertex-disjoint, for otherwise we have a short odd-cycle in $G$.

Let us call $G'$, the $12$-vertex subgraph induced by the vertices of $C$, $P^1$ and $P^5$. Notice that $G'$ also contains $P^0$, $P^2$, $P^3$ and $P^4$. Moreover, the vertices and edges of $C$, $P^1$ and $P^5$ form a $7$-odd $K_4$, more precisely, a $(1,2,4)$-odd-$K_4$. Now, because of the odd-girth of $G$, and by Claim~\ref{clm:(<=3,<=4)}, one can check that the only possible third neighbour of $v_1$ in $G'$, if any, is $v_4$ (and vice-versa). Furthermore, within the set $\{x_1,x_2,x_3,y_1,y_2,y_3\}$ of vertices, the only possible edges are those of the form $x_iy_i$ with $i=1,2,3$. (See Figure~\ref{fig:clm:no-6cycle}(i) for an illustration of $G'$.) We claim that $G'$ must be a subgraph of $C(12,5)$. Indeed, even if these four optional edges are present, we obtain an isomorphic copy of $C(12,5)$: consider the isomorphism from $C(12,5)$ to this graph, given by $\{0:v_0,1:v_4,2:v_2,3:v_3,4:y_2,5:x_1,6:v_5,7:v_1,8:v_3,9:y_3,10:x_2,11:y_1\}$. (See Figure~\ref{fig:clm:no-6cycle}(ii) for an illustration of $C(12,5)$.) Thus, the circular chromatic number of $G'$ is at most $12/5$, implying that $G'$ has a homomorphism to $C_5$.

\begin{figure}[!htpb]
\centering
\subfigure[The subgraph $G'$ of $G$. Dashed edges may or may not exist.]{\scalebox{0.8}{\begin{tikzpicture}[join=bevel,inner sep=0.6mm,scale=0.7]

\node[draw,shape=circle] (v2) at (30:4) {\tiny${v_2}$};
\node[draw,shape=circle] (v1) at (90:4) {\tiny${v_1}$};
\node[draw,shape=circle] (v0) at (150:4) {\tiny$v_0$};
\node[draw,shape=circle] (v5) at (210:4) {\tiny$v_5$};
\node[draw,shape=circle] (v4) at (270:4) {\tiny$v_4$};
\node[draw,shape=circle] (v3) at (330:4) {\tiny$v_3$};

\node[draw,shape=circle] at (-1.5,1) (x1) {\tiny$x_1$};
\node[draw,shape=circle] at (0,1) (x2) {\tiny$x_2$};
\node[draw,shape=circle] at (1.5,1) (x3) {\tiny$x_3$};
\node[draw,shape=circle] at (-1.5,-1) (y1) {\tiny$y_1$};
\node[draw,shape=circle] at (0,-1) (y2) {\tiny$y_2$};
\node[draw,shape=circle] at (1.5,-1) (y3) {\tiny$y_3$};

\draw[-,line width=0.4mm] (v0)--(v1)--(v2)--(v3)--(v4)--(v5)--(v0)--(x1)--(x2)--(x3)--(v3);
\draw[-,line width=0.4mm] (v5)--(y1)--(y2)--(y3)--(v2);
\draw[dashed,line width=0.4mm] (x1)--(y1) (x2)--(y2) (x3)--(y3) (v1) .. controls +(-7,0) and +(-7,0) .. (v4);

\end{tikzpicture}}}\qquad
\subfigure[The graph $C(12,5)$.]{\scalebox{0.8}{\begin{tikzpicture}[join=bevel,inner sep=0.6mm]

\node[draw,shape=circle] (8) at (0,-.2) {\small{8}};
\node[draw,shape=circle] (1) at (1.5,-.2) {\small{1}};
\node[draw,shape=circle] (6) at (3,-.2) {\small 6};
\node[draw,shape=circle] (11) at (4.5,-.2) {\tiny{11}};
\node[draw,shape=circle] (4) at (6,-.2) {\small 4};
\node[draw,shape=circle] (9) at (7.5,-.2) {\small 9};
\node[draw,shape=circle] (2) at (0,2.5) {\small 2};
\node[draw,shape=circle] (7) at (1.5,2.5) {\small 7};
\node[draw,shape=circle] (0) at (3,2.5) {\small 0};
\node[draw,shape=circle] (5) at (4.5,2.5) {\small 5};
\node[draw,shape=circle] (10) at (6,2.5) {\tiny{10}};
\node[draw,shape=circle] (3) at (7.5,2.5) {\small 3};

\path (0,-1.8) node (x) {};

\draw[-,line width=0.4mm] (2)--(7)--(0)--(5)--(10)--(3)--(9)--(2)
                               (2)--(8)--(3) (8)--(1)--(7) (1)--(6)--(0)
                               (6)--(11)--(5) (11)--(4)--(10) (4)--(9);
      \end{tikzpicture}}}
\caption{Graphs used in the proof of Claim~\ref{clm:no6-cycle}.}
\label{fig:clm:no-6cycle}
\end{figure}

Hence, there is at least one vertex in $G$, that is not a vertex of $G'$. Using Claim~\ref{clm:(<=3,<=4)}, one can check that any vertex of $G$ not in $G'$ can be adjacent only to $x_2$ or to $y_2$ (or both) and that there are at most two such vertices, one a neighbor of $x_2$ another a neighbor of $y_2$. Without loss of generality (considering the symmetries of the graph), assume that $x_2$ has an additional neighbour, $v$. Then, either $v$ is also adjacent to $y_2$ (then $G$ has order~$13$), or $v$ and $y_2$ have a common neighbour, say $w$. If $v$ is adjacent to both $x_2$ and $y_2$, then the only edges induced by the set $\{x_1,x_2,x_3,y_1,y_2,y_3\}$ are edges in $G'$ (otherwise we have a short odd-cycle). But then, we exhibit a homomorphism of $G$ to $C_5$: map $x_3$, $y_1$, $v$ to $0$; $x_2$, $y_2$ to $1$; $v_1$, $x_1$, $y_3$ to $2$; $v_0$, $v_2$, $v_4$ to $3$; $v_3$ and $v_5$ to $4$. This is a contradiction. Thus, $v$ and $y_2$ have a common neighbour $w$, and both $v$ and $w$ are of degree~$2$. We now create a homomorphic image of $G$ by identifying $v$ with $y_2$ and $w$ with $x_2$. Then, this image of $G$ is again a subgraph of $C(12,5)$, and thus the circular chromatic number of $G$ is at most $12/5$, implying that $G$ has a homomorphism to $C_5$, a contradiction. This completes the proof of Claim~\ref{clm:no6-cycle}.
\smallqed\medskip

\begin{claim}\label{clm:no4-cycle}
$G$ contains no $4$-cycle.
\end{claim}
\claimproof Assume by contradiction that $G$ contains a $4$-cycle $C:tuvw$. As in the proof of Claim~\ref{clm:no6-cycle}, there must be two $5$-paths $P_{tv}:ta_1a_2a_3a_4v$ and $P_{uw}:ub_1b_2b_3b_4w$ connecting $t$ with $v$ and $u$ with $w$, respectively. Moreover, these two paths must be vertex-disjoint because of the odd-girth of $G$. Thus, the union of $C$, $P_{tv}$ and $P_{uw}$ forms a $(1,1,5)$-odd-$K_4$, $K$. By assumption on the odd-girth of $G$, any additional edge inside $V(K)$ must connect an internal vertex of $P_{tv}$ to an internal vertex of $P_{uw}$. But any such edge would either create a short odd-cycle or a $6$-cycle in $G$, the latter contradicting Claim~\ref{clm:no6-cycle}. Thus, the only edges in $V(K)$ are those of $K$. If there is no additional vertex in $G$, we have two $5$-threads in $G$, contradicting Claim~\ref{clm:nothread}; thus there is at least one additional vertex in $G$, say $x$, and perhaps a last vertex, $y$. Note that $t$, $u$, $v$ and $w$ are already, in $K$, degree~$3$-vertices with a second neighbourhood of size $4$, thus by Claim~\ref{clm:(<=3,<=4)} none of $t$, $u$, $v$, $w$, $a_1$, $a_4$, $b_1$ and $b_4$ are adjacent to any vertex not in $K$. Thus $N(x), N(y) \subseteq \{a_2, a_3, b_2, b_3, x, y\}$.
 
Assume that some vertex not in $K$ (say $x$) is adjacent to two vertices of $K$. Then, these two vertices must be a vertex on the $tu$-path of $K$ ($a_2$ or $a_3$, without loss of generality it is $a_2$) and a vertex on the $vw$-path of $K$ ($b_2$ or $b_3$). By the automorphism of $K$ that swaps $v$ and $w$ and reverses the $vw$-path, without loss of generality we can assume that the second neighbour of $x$ in $K$ is $b_3$. Then, we claim that $G$ has no further vertex. Indeed, if there is a last vertex $y$, since $a_2$ and $b_3$ have already three neighbours, $y$ must be adjacent to at least two vertices among $\{a_3,b_2,x\}$. If it is adjacent to both $a_3$ and $b_2$, we have a $6$-cycle, contradicting Claim~\ref{clm:no6-cycle}; otherwise, $y$ is of degree~$2$ but part of a $4$-cycle, implying that $G$ is not a core, a contradiction. Thus, $G$ has order~$13$ and no further edge. We create a homomorphic image of $G$ by identifying $x$ and $a_3$. The obtained graph is a subgraph of $C(12,5)$. Hence, $G$ has circular chromatic number at most $12/5$ and a homomorphism to $C_5$, a contradiction.

Thus, any vertex not in $K$ has at most one neighbour in $K$. Since $G$ has no $4$-thread, one vertex not in $K$ (say $x$) is adjacent to one of $a_2$ or $a_3$ (without loss of generality, say $a_2$), and the last vertex, $y$, is adjacent to one of $b_2$ and $b_3$ (as before, by the symmetries of $K$ we can assume it is $b_2$). Moreover, $x$ and $y$ must be adjacent, otherwise they both have degree~$1$. Also there is no further edge in $G$. But then, as before, we create a homomorphic image of $G$ by identifying $x$ with $b_2$ and $y$ with $a_2$. The resulting graph is a subgraph of $C(12,5)$, which again gives a contradiction. This completes the proof of Claim~\ref{clm:no4-cycle}. \smallqed\medskip

Claims~\ref{clm:no6-cycle} and~\ref{clm:no4-cycle} imply that $G$ has girth at least~$7$. To complete the proof, we note that since $G$ has no homomorphism to $C_5$, it also has no homomorphism to $C_7$. Thus, by Theorem~\ref{thm:gerards}, $G$ must contain either an odd-$K_3^2$ or an odd-$K_4$ of odd-girth at least~$7$. Since such an odd-$K_3^2$ must have at least $18$ vertices, $G$ contains an odd-$K_4$. Let $H$ be such an odd-$K_4$ of $G$. Since in fact its \emph{girth} is at least~$7$, by Proposition~\ref{prop:oddK4} it has at least $12$ vertices. We consider three cases, depending on the order of $H$.

\paragraph{Case 1.} Assume that $|V(H)|=12$. By Proposition~\ref{prop:oddK4}, $H$ is a $7$-odd-$K_4$, in fact, it is an $(a,b,c)$-odd-$K_4$ with $a+b+c=7$. Since $H$ has girth at least~$7$, we have only two possibilities: (i) $(a,b,c)=(2,2,3)$ or (ii) $(a,b,c)=(1,3,3)$. 
Each of these two possible odd-$K_4$'s is of diameter~$4$. Furthermore, in each case, the pairs at distance~$4$ are isomorphic (each pair consists of two internal vertices of two disjoint threads of length~$3$).
If $G$ contains one more vertex than $H$, say $x$, then either $x$ is of degree at most 1, which contradicts Claim~\ref{clm:core} or it creates a cycle of length at most~6, which is not possible due to the girth of $G$. Thus $G$ has exactly two more vertices, say $x$ and $y$ and each of them is adjacent to at most one vertex in $H$. These two neighbours of $x$ and $y$ are at distance~4 in $H$ and $x$ is adjacent to $y$. 
There are only two non-isomorphic such possibilities.  A homomorphism to $C_5$ in each of these two possibilities is given in Figure~\ref{fig:case1}.

\begin{figure}[!htpb]
\centering
\subfigure[$(a,b,c)=(2,2,3)$]{\scalebox{0.8}{\begin{tikzpicture}[join=bevel,inner sep=0.6mm]
\node[draw,shape=circle,fill] (x) at (0,0) {};
\draw (x) node[left=0.2cm] {$2$};
\node[draw,shape=circle,fill] (y) at (6,0) {};
\draw (y) node[right=0.2cm] {$0$};
\node[draw,shape=circle,fill] (z) at (3,5.2) {};
\draw (z) node[above=0.2cm] {$2$};
\node[draw,shape=circle,fill] (t) at (3,1.7) {};
\draw (t) node[above right=0.1cm] {$0$};
\node[draw,shape=circle,fill] (a1) at (3,3.4) {};
\draw (a1) node[left=0.2cm] {$1$};
\node[draw,shape=circle,fill] (ap1) at (3,0) {};
\draw (ap1) node[above=0.2cm] {$1$};
\node[draw,shape=circle,fill] (b1) at (1.5,2.6) {};
\draw (b1) node[left=0.2cm] {$3$};
\node[draw,shape=circle,fill] (bp1) at (4.5,0.85) {};
\draw (bp1) node[above=0.2cm] {$1$};
\node[draw,shape=circle,fill] (c1) at (1,0.57) {};
\draw (c1) node[above=0.2cm] {$3$};
\node[draw,shape=circle,fill] (c2) at (2,1.13) {};
\draw (c2) node[above=0.2cm] {$4$};
\node[draw,shape=circle,fill] (cp1) at (4,3.5) {};
\draw (cp1) node[right=0.2cm] {$3$};
\node[draw,shape=circle,fill] (cp2) at (5,1.75) {};
\draw (cp2) node[right=0.2cm] {$4$};
\node[draw,shape=circle,fill] (d1) at (1.5,-1) {};
\draw (d1) node[left=0.2cm] {$4$};
\node[draw,shape=circle,fill] (d2) at (4.5,-1) {};
\draw (d2) node[right=0.2cm] {$3$};

\draw[-,line width=0.4mm] (x)--(y)--(z)--(x)--(t)--(y)  (t)--(z) (c1)--(d1)--(d2)--(cp2);
\end{tikzpicture}}}\qquad
\subfigure[$(a,b,c)=(1,3,3)$]{\scalebox{0.8}{\begin{tikzpicture}[join=bevel,inner sep=0.6mm]
\node[draw,shape=circle,fill] (x) at (0,0) {};
\draw (x) node[left=0.2cm] {$0$};
\node[draw,shape=circle,fill] (y) at (6,0) {};
\draw (y) node[right=0.2cm] {$4$};
\node[draw,shape=circle,fill] (z) at (3,5.2) {};
\draw (z) node[above=0.2cm] {$2$};
\node[draw,shape=circle,fill] (t) at (3,1.7) {};
\draw (t) node[above right=0.1cm] {$3$};
\node[draw,shape=circle,fill] (b1) at (1,1.7) {};
\draw (b1) node[left=0.2cm] {$4$};
\node[draw,shape=circle,fill] (b2) at (2,3.45) {};
\draw (b2) node[left=0.2cm] {$3$};
\node[draw,shape=circle,fill] (bp1) at (4,1.13) {};
\draw (bp1) node[above=0.2cm] {$2$};
\node[draw,shape=circle,fill] (bp2) at (5,0.57) {};
\draw (bp2) node[above=0.2cm] {$3$};
\node[draw,shape=circle,fill] (c1) at (1,0.57) {};
\draw (c1) node[above=0.2cm] {$1$};
\node[draw,shape=circle,fill] (c2) at (2,1.13) {};
\draw (c2) node[above=0.2cm] {$2$};
\node[draw,shape=circle,fill] (cp1) at (4,3.5) {};
\draw (cp1) node[right=0.2cm] {$1$};
\node[draw,shape=circle,fill] (cp2) at (5,1.75) {};
\draw (cp2) node[right=0.2cm] {$0$};
\node[draw,shape=circle,fill] (d1) at (2.5,-1) {};
\draw (d1) node[left=0.2cm] {$3$};
\node[draw,shape=circle,fill] (d2) at (4.25,-1) {};
\draw (d2) node[right=0.2cm] {$4$};

\draw[-,line width=0.4mm] (x)--(y)--(z)--(x)--(t)--(y)  (t)--(z) (c2)--(d1)--(d2)--(cp2);
\end{tikzpicture}}}\caption{The two graphs considered in Case~$1$.}
\label{fig:case1}
\end{figure}

\paragraph{Case 2.} Assume that $|V(H)|=13$. Let the lengths of the four odd-cycles of $H$ be $l_1$, $l_2$, $l_3$ and $l_4$. Using the fact that $l_1+l_2+l_3+l_4=2|E(H)|=30$ we conclude that three of these cycles are of length~$7$ and the last one is of length~$9$. Let $a$, $b$ and $c$ be the lengths of the three maximal threads of $H$ whose union forms the $9$-cycle. Let $a'$, $b'$ and $c'$ be the respective lengths of the corresponding pairwise disjoint threads.
Then, we have $a'+b'+c=7$ and $a+b+c=9$, which implies that $a'+b'=a+b-2$. Together with the similar equalities $a'+c'=a+c-2$ and $b'+c'=b+c-2$, we obtain that $a=a'+1$, $b=b'+1$ and $c=c'+1$. By the symmetries, we may assume that $a\leq b \leq c$. Considering that there are cycles of lengths $2a'+2b'+2$, $2a'+2c'+2$ and $2b'+2c'+2$ in $H$, we have either (i) $a'=b'=c'=2$ or (ii) $a'=1$, $b'=2$ and $c'=3$. In the first case, the graph is of diameter~$4$ and thus there is no room for a $14$th vertex; in the second case, there is a unique pair at distance~$5$, which leads us to a unique possible graph on $14$ vertices, by connecting a $14$th vertex to these two diametral vertices. Homomorphisms to $C_5$ for each of these graphs are given in Figure~\ref{fig:case2}.

\begin{figure}[!htpb]
\centering
\subfigure[$(a',b',c')=(2,2,2)$]{\scalebox{0.8}{\begin{tikzpicture}[join=bevel,inner sep=0.6mm]
\node[draw,shape=circle,fill] (x) at (0,0) {};
\draw (x) node[left=0.2cm] {$0$};
\node[draw,shape=circle,fill] (y) at (6,0) {};
\draw (y) node[right=0.2cm] {$3$};
\node[draw,shape=circle,fill] (z) at (3,5.2) {};
\draw (z) node[above=0.2cm] {$1$};
\node[draw,shape=circle,fill] (t) at (3,1.7) {};
\draw (t) node[above right=0.1cm] {$3$};
\node[draw,shape=circle,fill] (a1) at (3,4) {};
\draw (a1) node[left=0.2cm] {$0$};
\node[draw,shape=circle,fill] (a2) at (3,2.9) {};
\draw (a2) node[left=0.2cm] {$4$};
\node[draw,shape=circle,fill] (ap1) at (3,0) {};
\draw (ap1) node[above=0.2cm] {$4$};
\node[draw,shape=circle,fill] (b1) at (1,1.7) {};
\draw (b1) node[left=0.2cm] {$1$};
\node[draw,shape=circle,fill] (b2) at (2,3.45) {};
\draw (b2) node[left=0.2cm] {$0$};
\node[draw,shape=circle,fill] (bp1) at (4.5,0.85) {};
\draw (bp1) node[above=0.2cm] {$2$};
\node[draw,shape=circle,fill] (c1) at (1,0.57) {};
\draw (c1) node[above=0.2cm] {$1$};
\node[draw,shape=circle,fill] (c2) at (2,1.13) {};
\draw (c2) node[above=0.2cm] {$2$};
\node[draw,shape=circle,fill] (cp1) at (4.5,2.6) {};
\draw (cp1) node[right=0.2cm] {$2$};

\draw[-,line width=0.4mm] (x)--(y)--(z)--(x)--(t)--(y)  (t)--(z);
\end{tikzpicture}}}\qquad
\subfigure[$(a',b',c')=(1,2,3)$]{\scalebox{0.8}{\begin{tikzpicture}[join=bevel,inner sep=0.6mm]
\node[draw,shape=circle,fill] (x) at (0,0) {};
\draw (x) node[left=0.2cm] {$1$};
\node[draw,shape=circle,fill] (y) at (6,0) {};
\draw (y) node[right=0.2cm] {$0$};
\node[draw,shape=circle,fill] (z) at (3,5.2) {};
\draw (z) node[above=0.2cm] {$2$};
\node[draw,shape=circle,fill] (t) at (3,1.7) {};
\draw (t) node[above right=0.1cm] {$0$};
\node[draw,shape=circle,fill] (a1) at (3,3.4) {};
\draw (a1) node[left=0.2cm] {$1$};
\node[draw,shape=circle,fill] (b1) at (1,1.7) {};
\draw (b1) node[left=0.2cm] {$2$};
\node[draw,shape=circle,fill] (b2) at (2,3.45) {};
\draw (b2) node[left=0.2cm] {$1$};
\node[draw,shape=circle,fill] (bp1) at (4.5,0.85) {};
\draw (bp1) node[above=0.2cm] {$1$};
\node[draw,shape=circle,fill] (c1) at (0.75,0.425) {};
\draw (c1) node[above=0.2cm] {$2$};
\node[draw,shape=circle,fill] (c2) at (1.5,0.85) {};
\draw (c2) node[below=0.2cm] {$3$};
\node[draw,shape=circle,fill] (c3) at (2.25,1.275) {};
\draw (c3) node[above=0.2cm] {$4$};
\node[draw,shape=circle,fill] (cp1) at (4,3.5) {};
\draw (cp1) node[right=0.2cm] {$3$};
\node[draw,shape=circle,fill] (cp2) at (5,1.75) {};
\draw (cp2) node[right=0.2cm] {$4$};
\node[draw,shape=circle,fill] (d) at (0.75,4.5) {};
\draw (d) node[above=0.2cm] {$2$};

\draw[-,line width=0.4mm] (x)--(y)--(z)--(x)--(t)--(y)  (t)--(z) (c2)--(d)--(cp1);
\end{tikzpicture}}}\caption{The two graphs considered in Case~$2$.}
\label{fig:case2}
\end{figure}

\paragraph{Case 3.} Assume that $|V(H)|=14$. Again, let $l_1$, $l_2$, $l_3$ and $l_4$ be the lengths of the four odd-cycles of $H$. Then, we have $l_1+l_2+l_3+l_4=2|E(H)|=32$, and either $(I)$ $l_1=l_2=l_3=7$ and $l_4=11$, or $(II)$ $l_1=l_2=7$ and $l_3=l_4=9$. 

For $(I)$, we define $a$, $b$, $c$, $a'$, $b'$ and $c'$ similarly as in Case~$2$.
It is then implied that $a=a'+2$, $b=b'+2$ and $c=c'+2$. Since $a',b', c'\geq 1$ we have $a,b,c, \geq 3$ and since $a+b+c=11$ we have two possibilities: either $a=b=3$, $c=5$ or $a=3$, $b=c=4$. In the case $(a,b,c)=(3,3,5)$, the vertex of degree~3 of $H$ which is not on the face of size~11 has already five distinct vertices at distance~2, contradicting Claim~\ref{clm:(<=3,<=4)}. In the case $(a,b,c)=(3,4,4)$, $H$ is already of diameter $5$, thus $H$ is the same as $G$. For this case a $C_5$-colouring is given in Figure~\ref{fig:case3a}.

\begin{figure}[!htpb]
  \centering
\scalebox{0.8}{\begin{tikzpicture}[join=bevel,inner sep=0.6mm]
\node[draw,shape=circle,fill] (x) at (0,0) {};
\draw (x) node[left=0.2cm] {$0$};
\node[draw,shape=circle,fill] (y) at (6,0) {};
\draw (y) node[right=0.2cm] {$1$};
\node[draw,shape=circle,fill] (z) at (3,5.2) {};
\draw (z) node[above=0.2cm] {$1$};
\node[draw,shape=circle,fill] (t) at (3,1.7) {};
\draw (t) node[above right=0.1cm] {$4$};
\node[draw,shape=circle,fill] (a1) at (3,4) {};
\draw (a1) node[left=0.2cm] {$2$};
\node[draw,shape=circle,fill] (a2) at (3,2.9) {};
\draw (a2) node[left=0.2cm] {$3$};
\node[draw,shape=circle,fill] (b1) at (0.75,1.3) {};
\draw (b1) node[left=0.2cm] {$4$};
\node[draw,shape=circle,fill] (b2) at (1.5,2.6) {};
\draw (b2) node[left=0.2cm] {$3$};
\node[draw,shape=circle,fill] (b3) at (2.25,3.9) {};
\draw (b3) node[left=0.2cm] {$2$};
\node[draw,shape=circle,fill] (bp1) at (4.5,0.85) {};
\draw (bp1) node[above=0.2cm] {$0$};
\node[draw,shape=circle,fill] (c1) at (0.75,0.425) {};
\draw (c1) node[above=0.2cm] {$1$};
\node[draw,shape=circle,fill] (c2) at (1.5,0.85) {};
\draw (c2) node[below=0.2cm] {$2$};
\node[draw,shape=circle,fill] (c3) at (2.25,1.275) {};
\draw (c3) node[above=0.2cm] {$3$};
\node[draw,shape=circle,fill] (cp1) at (4.5,2.6) {};
\draw (cp1) node[right=0.2cm] {$0$};

\draw[-,line width=0.4mm] (x)--(y)--(z)--(x)--(t)--(y)  (t)--(z);
\end{tikzpicture}}
\caption{The first graph considered in Case~$3$, with $(a,b,c)=(3,4,4)$.}
\label{fig:case3a}
\end{figure}

For $(II)$, assuming $c'$ is the length of the common thread of the two cycles of length~$9$, a similar calculation as in Case~$2$ implies that $a=a'$, $b=b'$ and $c'=c+2$. 
The possibilities are (i) $a=1$, $b=3$ and $c=3$, (ii) $a=1$, $b=4$ and $c=2$, (iii) $a=2$, $b=2$ and $c=3$, (iv) $a=2$, $b=3$ and $c=2$, and (v) $a=b=3$ and $c=1$. As in each case, each pair of vertices is on a common cycle of length at most~$10$, and since the girth is at least~$7$, $G$ must be isomorphic to $H$. Then, a mapping to $C_5$ can readily be found (see Figure~\ref{fig:case3b}). This case completes the proof as there is no counterexample to the statement of the theorem.
\end{proof}
\begin{figure}[!htpb]
\centering
\subfigure[$(a,b,c)=(1,3,3)$]{\scalebox{0.8}{\begin{tikzpicture}[join=bevel,inner sep=0.6mm]
\node[draw,shape=circle,fill] (x) at (0,0) {};
\draw (x) node[left=0.2cm] {$1$};
\node[draw,shape=circle,fill] (y) at (6,0) {};
\draw (y) node[right=0.2cm] {$2$};
\node[draw,shape=circle,fill] (z) at (3,5.2) {};
\draw (z) node[above=0.2cm] {$4$};
\node[draw,shape=circle,fill] (t) at (3,1.7) {};
\draw (t) node[above right=0.1cm] {$0$};
\node[draw,shape=circle,fill] (b1) at (1,1.7) {};
\draw (b1) node[left=0.2cm] {$2$};
\node[draw,shape=circle,fill] (b2) at (2,3.45) {};
\draw (b2) node[left=0.2cm] {$3$};
\node[draw,shape=circle,fill] (bp1) at (4,1.13) {};
\draw (bp1) node[above=0.2cm] {$4$};
\node[draw,shape=circle,fill] (bp2) at (5,0.57) {};
\draw (bp2) node[above=0.2cm] {$3$};
\node[draw,shape=circle,fill] (c1) at (1,0.57) {};
\draw (c1) node[above=0.2cm] {$0$};
\node[draw,shape=circle,fill] (c2) at (2,1.13) {};
\draw (c2) node[above=0.2cm] {$1$};
\node[draw,shape=circle,fill] (cp1) at (3.6,4.16) {};
\draw (cp1) node[right=0.2cm] {$0$};
\node[draw,shape=circle,fill] (cp2) at (4.2,3.12) {};
\draw (cp2) node[right=0.2cm] {$1$};
\node[draw,shape=circle,fill] (cp3) at (4.8,2.08) {};
\draw (cp3) node[right=0.2cm] {$0$};
\node[draw,shape=circle,fill] (cp4) at (5.4,1.04) {};
\draw (cp4) node[right=0.2cm] {$1$};

\draw[-,line width=0.4mm] (x)--(y)--(z)--(x)--(t)--(y)  (t)--(z);
\end{tikzpicture}}}\qquad
\subfigure[$(a,b,c)=(1,4,2)$]{\scalebox{0.8}{\begin{tikzpicture}[join=bevel,inner sep=0.6mm]
\node[draw,shape=circle,fill] (x) at (0,0) {};
\draw (x) node[left=0.2cm] {$2$};
\node[draw,shape=circle,fill] (y) at (6,0) {};
\draw (y) node[right=0.2cm] {$3$};
\node[draw,shape=circle,fill] (z) at (3,5.2) {};
\draw (z) node[above=0.2cm] {$0$};
\node[draw,shape=circle,fill] (t) at (3,1.7) {};
\draw (t) node[above right=0.1cm] {$4$};
\node[draw,shape=circle,fill] (b1) at (0.75,1.3) {};
\draw (b1) node[left=0.2cm] {$1$};
\node[draw,shape=circle,fill] (b2) at (1.5,2.6) {};
\draw (b2) node[left=0.2cm] {$0$};
\node[draw,shape=circle,fill] (b3) at (2.25,3.9) {};
\draw (b3) node[left=0.2cm] {$1$};
\node[draw,shape=circle,fill] (bp1) at (3.75,1.275) {};
\draw (bp1) node[above=0.2cm] {$0$};
\node[draw,shape=circle,fill] (bp2) at (4.5,0.85) {};
\draw (bp2) node[above=0.2cm] {$1$};
\node[draw,shape=circle,fill] (bp3) at (5.25,0.425) {};
\draw (bp3) node[above=0.2cm] {$2$};
\node[draw,shape=circle,fill] (c1) at (1.5,0.85) {};
\draw (c1) node[above=0.2cm] {$3$};
\node[draw,shape=circle,fill] (cp1) at (3.75,3.9) {};
\draw (cp1) node[right=0.2cm] {$4$};
\node[draw,shape=circle,fill] (cp2) at (4.5,2.6) {};
\draw (cp2) node[right=0.2cm] {$0$};
\node[draw,shape=circle,fill] (cp3) at (5.25,1.3) {};
\draw (cp3) node[right=0.2cm] {$4$};

\draw[-,line width=0.4mm] (x)--(y)--(z)--(x)--(t)--(y)  (t)--(z);
\end{tikzpicture}}}\qquad
\subfigure[$(a,b,c)=(2,2,3)$]{\scalebox{0.8}{\begin{tikzpicture}[join=bevel,inner sep=0.6mm]
\node[draw,shape=circle,fill] (x) at (0,0) {};
\draw (x) node[left=0.2cm] {$0$};
\node[draw,shape=circle,fill] (y) at (6,0) {};
\draw (y) node[right=0.2cm] {$3$};
\node[draw,shape=circle,fill] (z) at (3,5.2) {};
\draw (z) node[above=0.2cm] {$3$};
\node[draw,shape=circle,fill] (t) at (3,1.7) {};
\draw (t) node[above right=0.1cm] {$1$};
\node[draw,shape=circle,fill] (a1) at (3,3.4) {};
\draw (a1) node[left=0.2cm] {$2$};
\node[draw,shape=circle,fill] (ap1) at (3,0) {};
\draw (ap1) node[above=0.2cm] {$4$};
\node[draw,shape=circle,fill] (b1) at (1.5,2.6) {};
\draw (b1) node[left=0.2cm] {$4$};
\node[draw,shape=circle,fill] (bp1) at (4.5,0.85) {};
\draw (bp1) node[above=0.2cm] {$2$};
\node[draw,shape=circle,fill] (c1) at (1,0.57) {};
\draw (c1) node[above=0.2cm] {$1$};
\node[draw,shape=circle,fill] (c2) at (2,1.13) {};
\draw (c2) node[above=0.2cm] {$0$};
\node[draw,shape=circle,fill] (cp1) at (3.6,4.16) {};
\draw (cp1) node[right=0.2cm] {$4$};
\node[draw,shape=circle,fill] (cp2) at (4.2,3.12) {};
\draw (cp2) node[right=0.2cm] {$0$};
\node[draw,shape=circle,fill] (cp3) at (4.8,2.08) {};
\draw (cp3) node[right=0.2cm] {$1$};
\node[draw,shape=circle,fill] (cp4) at (5.4,1.04) {};
\draw (cp4) node[right=0.2cm] {$2$};

\draw[-,line width=0.4mm] (x)--(y)--(z)--(x)--(t)--(y)  (t)--(z);
\end{tikzpicture}}}\qquad
\subfigure[$(a,b,c)=(2,3,2)$]{\scalebox{0.8}{\begin{tikzpicture}[join=bevel,inner sep=0.6mm]
\node[draw,shape=circle,fill] (x) at (0,0) {};
\draw (x) node[left=0.2cm] {$0$};
\node[draw,shape=circle,fill] (y) at (6,0) {};
\draw (y) node[right=0.2cm] {$2$};
\node[draw,shape=circle,fill] (z) at (3,5.2) {};
\draw (z) node[above=0.2cm] {$1$};
\node[draw,shape=circle,fill] (t) at (3,1.7) {};
\draw (t) node[above right=0.1cm] {$3$};
\node[draw,shape=circle,fill] (a1) at (3,3.4) {};
\draw (a1) node[left=0.2cm] {$2$};
\node[draw,shape=circle,fill] (ap1) at (3,0) {};
\draw (ap1) node[above=0.2cm] {$1$};
\node[draw,shape=circle,fill] (b1) at (1,1.7) {};
\draw (b1) node[left=0.2cm] {$1$};
\node[draw,shape=circle,fill] (b2) at (2,3.45) {};
\draw (b2) node[left=0.2cm] {$0$};
\node[draw,shape=circle,fill] (bp1) at (4,1.13) {};
\draw (bp1) node[above=0.2cm] {$4$};
\node[draw,shape=circle,fill] (bp2) at (5,0.57) {};
\draw (bp2) node[above=0.2cm] {$3$};
\node[draw,shape=circle,fill] (c1) at (1.5,0.85) {};
\draw (c1) node[above=0.2cm] {$4$};
\node[draw,shape=circle,fill] (cp1) at (3.75,3.9) {};
\draw (cp1) node[right=0.2cm] {$0$};
\node[draw,shape=circle,fill] (cp2) at (4.5,2.6) {};
\draw (cp2) node[right=0.2cm] {$4$};
\node[draw,shape=circle,fill] (cp3) at (5.25,1.3) {};
\draw (cp3) node[right=0.2cm] {$3$};

\draw[-,line width=0.4mm] (x)--(y)--(z)--(x)--(t)--(y)  (t)--(z);
\end{tikzpicture}}}\qquad
\subfigure[$(a,b,c)=(3,3,1)$]{\scalebox{0.8}{\begin{tikzpicture}[join=bevel,inner sep=0.6mm]
\node[draw,shape=circle,fill] (x) at (0,0) {};
\draw (x) node[left=0.2cm] {$4$};
\node[draw,shape=circle,fill] (y) at (6,0) {};
\draw (y) node[right=0.2cm] {$3$};
\node[draw,shape=circle,fill] (z) at (3,5.2) {};
\draw (z) node[above=0.2cm] {$1$};
\node[draw,shape=circle,fill] (t) at (3,1.7) {};
\draw (t) node[above right=0.1cm] {$0$};
\node[draw,shape=circle,fill] (a1) at (3,4) {};
\draw (a1) node[left=0.2cm] {$0$};
\node[draw,shape=circle,fill] (a2) at (3,2.9) {};
\draw (a2) node[left=0.2cm] {$1$};
\node[draw,shape=circle,fill] (ap1) at (2,0) {};
\draw (ap1) node[above=0.2cm] {$0$};
\node[draw,shape=circle,fill] (ap2) at (4,0) {};
\draw (ap2) node[above=0.2cm] {$4$};
\node[draw,shape=circle,fill] (b1) at (1,1.7) {};
\draw (b1) node[left=0.2cm] {$3$};
\node[draw,shape=circle,fill] (b2) at (2,3.45) {};
\draw (b2) node[left=0.2cm] {$2$};
\node[draw,shape=circle,fill] (bp1) at (4,1.13) {};
\draw (bp1) node[above=0.2cm] {$1$};
\node[draw,shape=circle,fill] (bp2) at (5,0.57) {};
\draw (bp2) node[above=0.2cm] {$2$};
\node[draw,shape=circle,fill] (cp1) at (4,3.5) {};
\draw (cp1) node[right=0.2cm] {$0$};
\node[draw,shape=circle,fill] (cp2) at (5,1.75) {};
\draw (cp2) node[right=0.2cm] {$4$};

\draw[-,line width=0.4mm] (x)--(y)--(z)--(x)--(t)--(y)  (t)--(z);
\end{tikzpicture}}}
\caption{Five graphs considered in Case~$3$.}
\label{fig:case3b}
\end{figure}

We now deduce the following consequence of Theorem~\ref{thm:15} and Proposition~\ref{prop:etaIncreasing}, that improves the known lower bounds on $\eta(4,C_5)$ and $\eta(4,C_3)$ (noting that for larger values of $k$, the bound of Corollary~\ref{cor:eta(k,l)>=4k} is already stronger).

\begin{corollary}\label{cor:eta(4,1)}
We have $\eta(4,C_3)\geq\eta(4,C_5)\geq 17$.
\end{corollary}

\section{Concluding remarks}\label{sec:conclu}

In this work, we have started investigating the smallest order of a $C_{2\ell +1}$-critical graph of odd-girth $2k+1$. We have determined a number of previously unknown values, in particular we showed that a smallest $C_5$-critical graph of odd-girth~7 is of order~15. In contrast to the result of Chv\'atal on the uniqueness of the smallest triangle-free 4-chromatic graph~\cite{C74}, we have found more than one such graph: Gordon Royle showed computationally, that there are exactly eleven such graphs (private communication, 2016) and Figure~\ref{fig:2graphsOrder15} shows three of them.

Regarding Table~\ref{table}, we do not know the growth rate in each row of the table. Perhaps it is quadratic; that would be to say that for a fixed $\ell$, $\eta(k,C_{2\ell+1})=\Theta(k^2)$. This is indeed true for $\ell=1$, as proved in \cite{N99}.

Our last remark is about Theorem~\ref{thm:columns}. We think that for any given $k$, Theorem~\ref{thm:columns} covers all values of $k$ for which $\eta(k,C_{2\ell+1})=4k$.

A similar problem of interest is to ask for the order of a smallest $C_{2\ell +1}$-critical graph of girth $g$. This question has not been studied much. In this formulation, the Gr\"otzsch graph is the smallest $C_3$-critical graph of girth~4. When girth~5 is considered, then again a computational work of Gordon Royle showed that 21 vertices are needed and again the answer is not unique~\cite{MO}. See~\cite{EG19} for further studies.

\end{document}